\documentclass[USenglish]{article}	

\usepackage[utf8]{inputenc}				
\usepackage[big,online]{dgruyter}	
\usepackage{lmodern} 
\usepackage{microtype}
\usepackage[numbers,square,sort&compress]{natbib}
\usepackage{ulem}

\theoremstyle{dgthm}
\newtheorem{theorem}{Theorem}
\newtheorem{corollary}[theorem]{Corollary}

\newtheorem{lemma}[theorem]{Lemma}

\theoremstyle{dgdef}

\newtheorem{remark}[theorem]{Remark}

\newcommand{\dmn}{{\Omega}}
\newcommand{\rob}{\partial \Omega}
\newcommand{\norm}[1]{\left\lVert#1\right\rVert}
\newcommand{\abs}[1]{\left|#1\right|}
\newcommand{\eps}{\varepsilon}
\newcommand{\TwoNorm}[2]{{\left\lVert#1\right\rVert}_{L^2(#2)}}
\newcommand{\HNorm}[2]{{\left\lVert#1\right\rVert}_{H^1(#2)}}
\newcommand{\HHNorm}[2]{{\left\lVert#1\right\rVert}_{H^2(#2)}}
\newcommand{\ii}{\mathbf{i}}

\newcommand{\T}{\mathcal{T}}
\newcommand{\nei}{\omega}
\newcommand{\Ctensor}{\mathfrak{C}}
\numberwithin{theorem}{section}
\numberwithin{equation}{section}

\begin{document}

  \startpage{1}
  \aop

\title{Multiscale Sub-grid Correction Method for 
       Time-Harmonic High-Frequency 
       Elastodynamics with Wavenumber Explicit  Bounds}
\runningtitle{Elastic Helmholtz Equation}

\author[1]{Donald L. Brown}
\author*[2]{Dietmar Gallistl}
\runningauthor{D. L. Brown and D. Gallistl}
\affil[1]{\protect\raggedright 
The Equity Engineering Group,
Applied Research and Development Division.
20600 Chagrin Blvd \#1200, Shaker Heights, OH 44122, USA}
\affil[2]{\protect\raggedright 
Friedrich-Schiller-Universit\"at Jena, Institut f\"ur Mathematik, 
Ernst-Abbe-Platz 2, 07743 Jena, Germany, 
e-mail: dietmar.gallistl[[AT symbol]]uni-jena.de}
	
	
\abstract{%
The simulation of the  elastodynamics equations  at high-frequency suffers from 
the well known pollution effect.
We present a Petrov--Galerkin multiscale sub-grid correction 
method that remains pollution free in natural  resolution and 
oversampling regimes. This is accomplished by generating corrections 
to coarse-grid spaces 
with supports determined by
oversampling lengths 
related to the $\log(k)$, $k$ being the wavenumber. Key to this method
are polynomial-in-$k$ bounds for stability constants and related 
inf-sup constants. To this end,
we establish polynomial-in-$k$ bounds for the 
elastodynamics stability constants in general Lipschitz domains
with radiation boundary conditions in $\mathbb{R}^3$. Previous
methods relied on variational techniques, Rellich identities,
and geometric constraints. In the context of elastodynamics,
these suffer from the need to hypothesize a Korn's inequality
on the boundary. 
The methods in this work are based  on 
boundary integral operators and estimation of Green's function's 
derivatives dependence on $k$  and do not require this extra
hypothesis. 
We also implemented numerical examples in two and
three dimensions to show the method eliminates pollution in
the natural resolution and oversampling regimes, as well as
performs well when  compared to standard Lagrange finite elements. 
}


\maketitle

\section{Introduction}
Modelling and simulating high-frequency wave propagation in complex
media is a computationally demanding process. The need in applications
to simulate wave propagation with accurate and robust numerical methods
is wide ranging. In acoustics applications such as the automotive 
and aerospace industries, the need to understand sound propagation is critical 
for vibration control and consumer comfort. For more complex 
mechanical media, the acoustic (Helmholtz) equation is not sufficient 
to describe the real propagation of signals and waves. This is the
case in subsurface seismic imaging applications, whereby attenuation
from the elastic properties must be taken into account. 
The simulation of accurate signal propagation through the subsurface
is utilized in calibrating the material properties in earth models 
and thus must be fast and robust. This is utilized in application 
domains ranging from environmental to petroleum exploration, 
and even mining engineering applications. 

It has been known for many years that at high-frequency, using
numerical methods in 
 Helmholtz type problems yields a pollution effect in the solution
 if the mesh parameter, $h$, 
 is able not to resolve the effects of high-frequency $k$. It is has
 been shown 
\cite{babuska1997pollution}
that using a finite stencil completely eliminating the
 pollution effect is impossible 
 for two or more space dimensions
 (while 
 the authors of \cite{babuska1997pollution} 
 were able to exhibit a pollution-free method in one space dimension).
 There
 are a wide range of methodologies and techniques for trying to
 combat the pollution error and here we mention only a few.  
 A plane wave Lagrange multiplier technique is utilized in 
 \cite{tezaur2006three} for the mid-frequency range. 
 Utilizing a Discontinuous Galerkin (DG) formulation for 
 both $h$ and $hp$ the authors develop methods for high-frequency
 in \cite{feng2009discontinuous,feng2011hp}.  In the  boundary
 integral method setting,  asymptotic  methods may be used 
 \cite{chandler2012numerical}. 
 {By going to a fully DG context
 and computing an optimal test space, the authors in 
 \cite{zitelli2011class} are able to obtain a pollution free method
 in one space dimension.}
 A breakthrough was achieved by relating the 
 polynomial order $p$ to the frequency $k$ in a logarithmic way 
 utilizing $hp$ methods in 
 \cite{melenk2010convergence,melenk2011wavenumber}. 
 In this work, we shall utilize a method based on the so-called 
 Local Orthogonal Decomposition (LOD) \cite{MP14}.
 
 The LOD method is based on utilizing quasi-interpolation operators 
 to build a fine scale or detail space. Then, by forcing orthogonality 
 an augmented coarse-space is constructed. However, the 
 supports
 of 
 these corrections to the original coarse-space are global. 
 To make the scheme computationally efficient a truncation to local 
 patch procedure is implemented. The errors of such a truncation can be 
 carefully tracked \cite{Henning.Morgenstern.Peterseim:2014,MP14}.
 The method originally was conceived to handle elliptic
 problems with very rough and multiscale coefficients. 
 However, since its inception it has seen generalizations to 
 semi-linear equations
 \cite{HMP13,HMP12}, oversampling methods \cite{HP13}, problems 
 with microstructure \cite{Brown.Peterseim:2014}, and parabolic
 problems \cite{maalqvist2015multiscale}, just to name a few.
For this paper, we utilize the method's effective ability to eliminate
pollution in optimal coarse-grid, $H$, and frequency, $k$, 
regimes 
\cite{BrownGallistlPeterseim,GallistlPeterseim2015,Peterseim2016}. 
Assuming polynomial-in-$k$ growth of stability and inf-sup constants 
of the continuous problem, and supposing patch truncations 
(oversampling) parameter $m$ of order $\log(k)$, we obtain a pollution 
free method in the resolution condition, $Hk\lesssim 1$ range. 

However, polynomial-in-$k$ growth of stability constants is not 
guaranteed as trapping domains may yield constants of order 
$\exp(k),$ for certain frequency values
\cite{betcke2011condition,chandler2009condition}.
The study and calculation of the stability and related inf-sup 
constants and their dependence on $k$ is a vivid area of active
research
that dates back to the early works
\cite{MorawetzLudwig1968,Morawetz1975,Vainberg1975}
as well as
\cite{FengSheen1994}
.
 In the acoustic setting, given certain geometric and
convexity conditions, the work of  
\cite{melenk1995generalized} gave  polynomial (constant) bounds  
in $k$.   Subsequent generalizations and extensions of these methods
in the Helmholtz setting can be found in 
\cite{cummings2006sharp,hetmaniuk2002fictitious,hetmaniuk2007stability,Spence2014,BaskinSpenceWunsch2016}.
In general, these methods rely on variational techniques that 
utilize the special  test function $x\cdot \nabla u$, Rellich
identities, geometric constraints, and various boundary condition 
constraints. These methods have also been extended to the
case of smooth weakly heterogeneous coefficients 
\cite{BrownGallistlPeterseim}, although with serious constraints on 
coefficients that can be considered. 
 
 Prior to this work, a key issue in using this variational and 
 Rellich identities method for the elastodynamics (elastic Helmholtz) 
 case is the fact that the stress terms arrive on certain boundary
 integrals. In standard elasticity analysis,  Korn's second inequality
 (see \cite{mclean2000strongly}) is employed to handle these terms in 
 the interior of the domain. However, in the fundamental work of 
 \cite{cummings2006sharp}, to obtain gradient lower bounds on these 
 stress terms a ``boundary'' Korn's inequality must be 
 \emph{conjectured} of the form
 \begin{align}\label{boundarykorn}
 	 	\norm{\frac{1}{2}\left(\nabla u+(\nabla u)^T \right)}_{L^2(\partial \Omega)}^2
 	 	\geq 
 	 	C\left(\norm{\nabla u}_{L^2(\partial \Omega)}^2 -
 	 	  \norm{u}_{L^2(\partial \Omega)}^2 \right),
 \end{align}
 for some boundary  
 $\partial \Omega$. Here $\Omega\subset \mathbb{R}^{d}, d=2,3.$
 Although such a result may seem reasonable it is as yet unproven 
 and must be assumed to obtain polynomial-in-$k$ bounds for the 
 elastodynamics case.

However, in the recent works of
  \cite{Esterhazy,melenk2010convergence,MelenkBIO,Spence2014},
  a new technique 
 is developed in the Helmholtz case without using the variational 
 techniques and instead relying on boundary integral operators, 
 estimates-in-$k$ of Green's functions, and Green's identity. 
 These techniques make minimal (Lipschitz domain) geometric
 assumptions on the domain, but must be used for radiation 
 (Robin type) boundary conditions. It is this technique that we 
 use in this work for the elastodynamics case with radiation boundary
 conditions. The main theoretical contribution of this work 
is to show methods in \cite{Esterhazy,melenk2010convergence,melenk2011wavenumber} can be extended to the case of elastodynamics 
 (elastic Helmholtz) and establish polynomial-in-$k$ bounds without the use of a conjectured boundary
 Korn's inequality \eqref{boundarykorn} in $\mathbb{R}^3$.
   We finally mention the independent work
  \cite{ChaumontNicaise2018},
  which establishes similar stability bounds
  (with sharper $k$-dependence but a slightly different geometric
   setting).

The proof of the polynomial bounds requires many auxiliary computations
and results, so we sketch the organization of the proof and paper 
as follows. In Section~\ref{elastodynamicsgoverning}, we present 
the problem set up of the elastodynamics equations with radiation
boundary conditions in the frequency domain and its related variational
forms.  We then introduce the Green's function, $G_k$ and related potential operators. 
In this mechanics context, $G_k$ is referred to as the Kupradze 
matrix \cite{kupradze1965potential,kupradze2012three}. 
 We define the Newton potential, $N_k$
and  present $k$-dependent estimates as in \cite{melenk2010convergence}. 
From here we obtain estimates for the single 
and double layer potentials, denoted  
$\widetilde{V}_{k},\widetilde{K}_{k},$ respectively.
This is accomplished by employing a technique from \cite{Spence2014}, where
layer potential estimates are obtained  for nontrivial $k$.
Then, we utilize standard a-priori and variational techniques,
along with the 
 ideas from 
 \cite{FengSheen1994,Esterhazy},
where  we  write the solution via single
 and double layer potentials
by   Green's identity. 
Using the estimates for the layer potentials 
we obtain an estimate for $u$ in terms of known boundary norms of
$u$ and $\sigma(u) \nu$. This gives polynomial-in-$k$ bounds as
well as the related inf-sup constant.

In Section~\ref{s:multiscale}, we present the ideas of the multiscale
sub-grid correction algorithm as in 
\cite{BrownGallistlPeterseim,GallistlPeterseim2015,Peterseim2016}.
Here we present the algorithm as well as the basic error analysis.
We implemented examples for both two and three dimensional examples
and see that the method performs well in handling pollution effects 
when compared to the standard Lagrange finite elements. 

Finally, the Appendix~\ref{NPSection} is devoted to the proof of the
main Newton potential estimates.
The $k$-dependence of derivatives of  Green's function 
(Kupradze matrix) 
is critical for the estimates in 
Section~\ref{elastodynamicsgoverning}. 
We use the Fourier techniques to estimate the Green's function 
similar to \cite{melenk2010convergence}.
The key observation, suggested by an anonymous reviewer in a previous version of the manuscript, 
is to notice the additional term of the Green's function has a cancellation in the leading order of the singularity. This observation greatly simplifies 
the analysis of a previous version by the authors.
The Green's function estimate then yields estimates for $N_k$.

\bigskip
Standard notation on complex-valued Lebesgue and Sobolev spaces
applies in this paper.
The real and imaginary part of $z\in\mathbb C$ reads
$\Re z$ and $\Im z$, respectively.
The imaginary unit is denoted by $\ii$ and
the bar denotes complex conjugation
$\bar z = \Re z - \ii \Im z$.
The analysis in this paper is explicit in the 
wavenumber $k$ and in the mesh-size parameters $H$ and $h$
of the finite element spaces.
Generic constants, often denoted by $C$ and possibly having different
values at different occurrences, are independent of
those parameters.
An inequality $A\leq CB$, will be frequently abbreviated by
$A\lesssim B$.
The notation $A\approx B$ means $A\lesssim B \lesssim A$.
We focus on the regime of large wave numbers $k$ and will sometimes 
use that $k$ is sufficiently large.

\section{Elastodynamic Equation in the  Frequency Domain}\label{elastodynamicsgoverning}

In this section, we will introduce the governing equations for the
three dimensional elastodynamic equation. Here we introduce the 
related variational form, as well as the critical  Korn 
(in the interior of the domain) and G\aa rding inequalities. 
We will then prove the wavenumber explicit bounds for the solution 
and show that we obtain frequency $k$-polynomial bounds of the 
solution given a radiation boundary condition on a bounded connected
Lipschitz domain. From this we are able to obtain an estimate for
the inf-sup stability constant utilized heavily in our multiscale 
numerical algorithm.

\subsection{Elastodynamics Governing Equations}\label{ss:governing}
 We now begin the problem setting. First, let 
 $\Omega\subset \mathbb{R}^3$ be a bounded open connected Lipschitz 
 domain in $\mathbb{R}^3$ and 
take 
$f\in L^2(\Omega)^3$
and $g\in L^2(\rob)^3$.
We suppose that $u$ satisfies the elastodynamic equation 
(time-harmonic elastic wave equation in the frequency domain) with  
radiation boundary conditions. 
We suppose that our material
satisfies a homogeneous isotropic stress tensor
\begin{equation}\label{isostress}
\sigma(u)=2 \mu \eps(u)+\lambda \text{tr}(\eps(u))=\mu (\nabla u+(\nabla u)^T)+\lambda \text{div}(u)I ,
\end{equation}
where the symmetric gradient is given by  
$\eps(u)=(\nabla u+(\nabla u)^T)/2$, $I$ is the identity matrix,
and superscript $T$ denotes transpose.
We let $u=(u^{(1)},u^{(2)},u^{(3)})$ be the solution of the
following governing equations
\begin{subequations}\label{mainPDE}
\begin{align}\label{main}
-\text{div}(\sigma(u))-k^2u&=f \text{ in } \dmn,\\
 \label{BCs}
\sigma(u) \nu +\ii k u&=g \text{ on } \rob
\end{align}
\end{subequations}
where $\nu$ is the outward normal.
Note that this is equivalent to the Lam\'{e} equations
\begin{align*}
 -\mu \Delta u -(\lambda+\mu )\nabla (\text{div}(u))-k^2u&=f \text{ in } \dmn,\\
\lambda \text{div}(u)  \nu+\mu (\nabla u+(\nabla u)^T) \nu +\ii k u&=g \text{ on } \rob. 
\end{align*}
Here $\lambda$ and $\mu$, are the Lam\'{e} constants with 
$\mu>0$ and $\lambda >-2\mu/3$ \cite{dahlberg1988,kenig1983boundary}. 
We require  further that $\lambda+2\mu>0$ for strong ellipticity 
\cite[p.~297]{mclean2000strongly}. 
The boundary condition is the elastodynamic analogue of the 
Robin boundary condition 
$\nabla u \cdot \nu +\ii ku=g$ 
in the acoustic Helmholtz context.
In elasticity, a traction boundary condition is  required, here we 
define the traction conormal as
\begin{align}\label{traction}
\frac{\partial u}{\partial \nu}:=\sigma(u)\nu =\lambda \text{div}(u)  \nu+\mu (\nabla u+(\nabla u)^T) \nu,
\end{align}
which for notational brevity we will often denote $\partial_{\nu} u$. 
The corresponding  variational form can be arrived at via the Betti 
formula \cite{Hsiao}, multiplying  \eqref{mainPDE} by $\bar{v}$ and
integrating we obtain 
\begin{align*}
 -\int_{\dmn} \text{div}(\sigma(u))\cdot\bar{v}dx
 -\int_{\dmn}k^2u \cdot\bar{v}dx
 = 
 \int_{\dmn} \sigma(u):\nabla\bar{v}dx
 -\int_{\partial \dmn} \sigma(u)\nu  \bar{v}ds
 -\int_{\dmn}k^2u\cdot\bar{v}dx\\
 =
  \int_{\dmn}\left(\mu (\nabla u+(\nabla u)^T)+\lambda \text{div}(u)I\right):\nabla \bar{v}dx
 -\int_{\partial \dmn}(g-\ii ku)  \cdot\bar{v}ds
 -\int_{\dmn}k^2u\cdot\bar{v}dx.
\end{align*}
Here we have  full contraction denoted by $A:B=\sum_{ij}A_{ij}B_{ij}$. 
Thus, using the identity 
$\text{div}(u)I:\nabla \bar{v}=\text{div}(u)\text{div}(\bar{v})$, 
and $2(\nabla u+(\nabla u)^T):\nabla \bar{v}=(\nabla u+(\nabla u)^T):(\nabla \bar{v}+(\nabla \bar{v})^T),$ 
we may write the following variational form.
Find $u\in (H^1(\dmn))^3,$ satisfying
\begin{align}\label{varformequation}
 \Phi_{}(u,\bar{v})=F(v) \text{ for all } \bar{v}\in (H^1(\dmn))^3,
 \end{align}
 where
 $\Phi_{}(u,\bar{v})=\Phi_{\dmn}(u,\bar{v})+\Phi_{\partial \dmn}(u,\bar{v})$ 
 and $F$
 are given by
 \begin{subequations} \label{varform}
 \begin{align}
 &\Phi_{\dmn}(u,\bar{v})
   =
  \int_{\dmn} \left(\lambda\text{div}(u)\text{div}(\bar{v})+{2\mu}\eps(u)
              : \eps(v)-k^2 u\cdot\bar{v}\right)dx\\
 &\Phi_{\partial \dmn}(u,\bar{v})
   =
  \ii k \int_{\partial \dmn}u\cdot\bar{v}ds, 
  \qquad 
   F(v)=\int_{\partial \dmn}g \cdot\bar{v} ds+\int_{\dmn}f\cdot\bar{v}dx .
\end{align}
\end{subequations}
Here $dx$ is the standard Lebesgue volume measure, and $ds$ the 
standard Lebesgue surface measure. 
We will denote $\norm{\cdot}_{L^2(\dmn)}$ and $\norm{\cdot}_{H^s(\dmn)}$,
$s=1,2$, to be the standard Sobolev norms. When there is no 
ambiguity we will not differentiate between
the vector norms and scalar norms. 

Recall, for the open bounded
Lipschitz domain $\Omega\subset\mathbb R^3$,
Korn's second inequality (in the interior) \cite{mclean2000strongly} 
\begin{equation}\label{korn_second}
\norm{\varepsilon(u)}_{L^2(\Omega)}^2
\geq 
c_{\mathrm{Korn}} \norm{\nabla u}_{L^2(\Omega)}^2 -
  C_{\mathrm{Korn}}   \norm{u}_{L^2(\Omega)}^2 .
\end{equation}
We have for some $C,c$ positive and independent of $k$, that
\begin{equation}
\begin{aligned}\label{Garding}
 \Re \Phi_{\dmn}(u,\bar{u})&=2 \mu \norm{\varepsilon(u)}_{L^2(\dmn)}^2+\lambda \norm{\text{div}(u)}_{L^2(\dmn)}^2-k^2 \norm{u}_{L^2(\dmn)}^2\\
 &\geq C \norm{\nabla u}_{L^2(\dmn)}^2-(c+k^2) \norm{u}_{L^2(\dmn)}^2.
\end{aligned}
\end{equation}
Indeed, we see that this is the G\aa rding's inequality  on $(H^1(\dmn))^3$.
With an argument analogous to that of
\cite[Sect.~8.1]{melenk1995generalized} one can show that 
\eqref{varformequation} is uniquely solvable.
We will make the $k$-dependence precise in
Theorem~\ref{mainesttheorem} below.

\subsection{Green's Functions and Potentials' Estimates}\label{LayerSection}

In this 
subsection,
we will introduce the ideas and notation to analyze the first and second layer potentials in elastodynamics. We will introduce the Green's function (in this setting it is the Kupradze matrix) and related boundary and Newton potentials. These, via Somigliana's formula (Green's identiy), give a representation for the solution of the elastodynamics equation.

We begin with some notation following the style introduced in  \cite{dahlberg1988} and \cite{MelenkBIO}.
As above, let $\Omega\subset \mathbb{R}^3$ be a bounded, open, connected Lipschitz  domain in $\mathbb{R}^3$. 
Finally, we suppose the convention 
\cite{MelenkBIO}
that $k\geq k_0>0$ for some fixed $k_0$.
As seen in Section~\ref{kgrowth}, 
Theorem~\ref{t:infsup}, $k_0$ is related to the Korn's second inequality
 constant $C_{\text{Korn}}$, and depends only on the domain and material parameters, but not on $k$. 
In general, we are considered here with high-frequency (large $k$) problems.
%
%
We now define the Green's function, or in this setting often referred to as the  Kupradze matrix, corresponding to \eqref{main}.
The Green's function is the fundamental solution to the equation
\begin{align}
-\text{div}\left(\sigma\left(G_{k}(x-y)\right)\right)-k^2G_{k}(x-y)&=\delta(x-y) \text{ in } \dmn,
\end{align}
or expanded out into the Lam\'{e} equations form
\begin{align}
-\mu \Delta G_{k}(x-y) -(\lambda+\mu )\nabla (\text{div}(G_{k}(x-y)))-k^2G_{k}(x-y)&=\delta(x-y) \text{ in } \dmn
\end{align}
for the Dirac distribution $\delta$.

\begin{remark}
	There are many ways to obtain the Green's function of the above system.
For $d=2,3$, the Green's function can be obtained as in \cite{kitahara2014boundary} via H{\"o}rmander's method of parametrics
and derived via Radon transforms in \cite{Wang441}.
For a complete treatment of this subject we refer the readers to  \cite{kupradze1965potential,kupradze2012three}. 
\end{remark}

Letting $k_1^2=k^2(\lambda+2\mu )^{-1} \, , \,  k_2^2=k^2\mu^{-1}$, for $i,j=1,2,3$,
we write the Green's function as
\begin{subequations}\label{GreensFunction.main}
\begin{align}
\label{GreensFunction}
( G_{k}(r))_{ij}&=\frac{1}{4 \pi \mu }\left(\delta_{ij}\frac{e^{\ii  k_2 r}}{r}+ \frac{1}{k_2^2}\frac{\partial^2 }{\partial x_{i}\partial x_j}
\left(\frac{e^{\ii k_2 r}}{r}-\frac{e^{\ii k_1 r}}{r}\right)\right),\\
\label{GreensFunction.different}
( G^{E}_{k}(r))_{ij}&=\frac{1}{4 \pi \mu }\left( \frac{1}{k_2^2}\frac{\partial^2 }{\partial x_{i}\partial x_j}
\left(\frac{e^{\ii k_2 r}}{r}-\frac{e^{\ii k_1 r}}{r}\right)\right),\\
\label{GreensFunction.same}
( G^{H}_{k}(r))_{ij}&=\frac{1}{4 \pi \mu}\left(\delta_{ij}\frac{e^{\ii  k_2 r}}{r}\right),
 \end{align}
 \end{subequations} 
where $r=|x|$, is the standard Euclidean distance in spherical coordinates. 
Note that the elastic rotational term \eqref{GreensFunction.different} is smoother than it first appears. Indeed, note that $\left(\frac{e^{\ii k_2 r}}{r}-\frac{e^{\ii k_1 r}}{r}\right)\to 0$ as $r\to 0$, so that the leading order singularity is cancelled.

 \begin{remark}
 	Note that a two-dimensional representation for the elastodynamics Green's function exists. Indeed, from \cite{kitahara2014boundary}, we have 
 	$$
 	( G_{k}(r))_{ij}=\frac{\ii}{4  \mu }\left(\delta_{ij}H_0^{(1)}(k_{2}r)+ \frac{1}{k_2^2}\frac{\partial^2 }{\partial x_{i}\partial x_j}
 	\left(H_0^{(1)}(k_{2}r)-H_0^{(1)}(k_{1}r)\right)\right),
 	$$
 where $H_0^{(1)}$ are zero order Hankel functions of the first kind. 
 We 
   expect
  that similar to \cite{melenk2010convergence},
 the following estimates hold for the two-dimensional, but leave this to future investigations due to the extra technical difficulties in handling Hankel functions as opposed to merely exponential functions of the form $\exp(\ii k r)/r$. 
 \end{remark}

The Newton potential for $f\in (L^2(\dmn))^3$ with compact support in $\mathbb{R}^3$, is given by
\begin{align}\label{newton.layerk}
 (N_{k}(f))(x)=\int_{\dmn}G_{{k}}(x-y)f(y)dy, \text{ for } x\in \dmn.
\end{align}
Here we use the notation to mean $(N_{k}(f))_i=\sum_{j=1}^3\int_{\dmn}(G_{k})_{ij}(x-\cdot)f_j dy, i=1,2,3,$ and similarly for the operators defined below. 
We further define the layer potential operators   as in  \cite{kenig1983boundary} and  \cite{MelenkBIO} for $\varphi\in (L^2(\rob))^3$ and $x\in \dmn$ as
\begin{align}\label{layerpotential}
\widetilde{V}_k(\varphi)(x):=\int_{\partial \dmn}G_{{k}}(x-y)\varphi(y)ds_{y}, \hspace{.1in} \widetilde{K}_k(\varphi)(x):=\int_{\partial \dmn}\partial_{\nu_y} G_{{k}}(x-y)\varphi(y)ds_{y},
\end{align}
here $\partial_{\nu_y},$ to denote the traction conormal,  \eqref{traction}, with respect to $y$.

We then have a representation formula for \eqref{mainPDE} in the interior 
of the domain given by \eqref{solutionrepresentation} below.
This is the so-called  Somigliana's formula (Green's identity in matrix form)
\cite[\S1.6.2, p.~25]{kitahara2014boundary}, which we briefly describe here. 
Denote the operator corresponding to elastodynamics as 
${\mathcal L}_{k}=-\mu \Delta  -(\lambda+\mu )\nabla \text{div}-k^2$.
From the Green's identity we have for $u$ a solution to \eqref{mainPDE} 
for $f=0$
\begin{align*}
 \int_{\dmn}\left(G_{k}(x-y){\mathcal L}_{k}u(y)
  -  
   {\mathcal L}_{k}
  G_{k}(x-y)u(y)   \right)dy
 =
 \int_{\partial \dmn}\left(-G_{k}(x-y)\partial_{\nu_y}u(y) 
 + \partial_{\nu_y}G_{k}(x-y)u(y)   \right)ds_y.
\end{align*} 
Hence, we have the solution representation as 
\begin{align*}
 u(x)=\int_{\partial \dmn}\left(G_{k}(x-y)\partial_{\nu_y}u(y)
    -  \partial_{\nu_y}G_{k}(x-y)u(y)   \right)ds_y, \, x\in \dmn.
\end{align*}
Finally,  \eqref{mainPDE}
 has a single layer and double layer potential representation of the form
\begin{align}\label{solutionrepresentation}
 u(x)&=\widetilde{V}_{k}(\sigma(u) \nu)(x)-\widetilde{K}_{k}(u)(x), \, x\in \dmn, 
\end{align}
where $\sigma(u) \nu$ is defined via \eqref{traction}. 

The above potentials satisfy the following $k$-dependent bounds. 

\begin{theorem}[Newton potential estimates]\label{maintheorem}
Let $f\in (L^2(\dmn))^3,$  
for the Newton potential \eqref{newton.layerk} we have the estimate
\begin{align}\label{NPHTWO}
 k^{-1}\HHNorm{N_{k}^{}(f)}{\dmn}+ \HNorm{N_{k}^{}(f)}{\dmn}
 + k\TwoNorm{N_{k}^{}(f)}{\dmn}\leq C \TwoNorm{f}{\dmn},
\end{align}
where $C>0$ is independent of $k$ and depends only on 
 $\Omega,\mu,\lambda$. 
\end{theorem}

\begin{proof}
See Appendix \ref{NPSection}.
\end{proof}
 
With the Newton potential estimate, Theorem \ref{maintheorem}, we can establish the following layer potential estimates. 

\begin{lemma}[layer potential estimates]\label{l:layerpotentials}
 The layer potential operators
 $\widetilde{V}_k$ and $\widetilde{K}_k$ defined in 
 \eqref{layerpotential} satisfy, for any
 $\varphi\in (L^2(\rob))^3$, the estimates
 \begin{align*}
 \|\widetilde{V}_k(\varphi)\|_{L^2(\Omega)}
 \leq C k^{-1/2} \|\varphi\|_{L^2(\partial\Omega)}
 \quad\text{and}\quad
  \|\widetilde{K}_k(\varphi)\|_{L^2(\Omega)}
 \leq C k^{1/2}\|\varphi\|_{L^2(\partial\Omega)}.
\end{align*}
\end{lemma}
\begin{proof}
 The proof is very similar to that of 
 \cite[Lemma~4.3]{Spence2014} and we briefly show it here 
 for completeness and self-consistent reading.
 Recall the Newton potential $N_k$ from \eqref{newton.layerk}
 whose adjoint $N_k'$ is given by
\begin{align*}
 (N_{k}'(f))(x)
 =
 \int_{\dmn}\overline{G_{{k}}(x-y)}f(y)dy, \text{ for } x\in \dmn,
\end{align*}
and satisfies $N'_k(f) = \overline{N_k(\bar{f})}$.
With the definitions \eqref{layerpotential}
this leads to the identity
\begin{equation*}
 \int_\Omega \widetilde{V}_k(\varphi)(x) \overline{f(x)}\,dx
 =
 \int_{\partial\Omega} \varphi(x) \overline{N_k'(f)(x)}\,ds_x.
\end{equation*}
The multiplicative trace inequality therefore shows
\begin{equation*}
 \int_\Omega \widetilde{V}_k(\varphi)(x) \overline{f(x)}\,dx
\leq
 \|\varphi\|_{L^2(\partial\Omega)} \|N_k'f\|_{L^2(\partial\Omega)}
\leq
\|\varphi\|_{L^2(\partial\Omega)} 
 \|N_k'f\|_{L^2(\Omega)}^{1/2}
 \|N_k'f\|_{H^1(\Omega)}^{1/2}.
\end{equation*}
The Newton potential estimate \eqref{NPHTWO} 
from Theorem~\ref{maintheorem} and the 
identity $N'_k(f) = \overline{N_k(\bar{f})}$
therefore yield
\begin{equation*}
 \int_\Omega \widetilde{V}_k(\varphi)(x) \overline{f(x)}\,dx
\leq C k^{-1/2}
\|\varphi\|_{L^2(\partial\Omega)}  \|f\|_{L^2(\Omega)}.
\end{equation*}
Choosing $f=\widetilde{V}_k(\varphi)$ gives the first 
asserted estimate.
Analogous computations for the double-layer potential
read
 \begin{equation*}
\int_\Omega \widetilde{K}_k(\varphi)(x) \overline{f(x)}\,dx
 \int_{\partial\Omega} 
     \varphi(x) \overline{\partial_\nu N_k'(f)(x)}\,ds_x 
\leq
 \|\varphi\|_{L^2(\partial\Omega)} 
 \|N_k'f\|_{H^1(\Omega)}^{1/2}
 \|N_k'f\|_{H^2(\Omega)}^{1/2}
\end{equation*}
and the combination with \eqref{NPHTWO} implies
\begin{equation*}
\int_\Omega \widetilde{K}_k(\varphi)(x) \overline{f(x)}\,dx
\leq
C k^{1/2}
\|\varphi\|_{L^2(\partial\Omega)} 
\|f\|_{L^2(\Omega)}.
\end{equation*}
The choice $f=\widetilde{K}_k(\varphi)$ then proves the 
second claimed estimate.
\end{proof}

\subsection{\texorpdfstring{$k$}{k}-Polynomial Growth of the Stability Constant}\label{kgrowth}

Here we present the main $k$-growth  estimates of this work.
We establish polynomial growth of the stability constant with respect 
to wavenumber $k$ for system \eqref{mainPDE}.
We will first require a few auxiliary lemmas similar to those 
obtained for the Helmholtz case in
\cite{Esterhazy,Spence2014}.
Throughout this section, we will use the
following notation
$$
\norm{u}_{1,k}:= \left(\TwoNorm{\nabla u}{\dmn}^2+k^2\TwoNorm{u}{\dmn}^2\right)^{1/2}.
$$

\begin{lemma}\label{l:cauchydata}
 Let $\dmn\subset\,\mathbb{R}^3$ be a bounded, 
 open, connected
 Lipschitz domain.
 Let $u\in (H^1(\dmn))^3$ be a solution to \eqref{mainPDE} with $f=0$
 with $g\in(L^2(\partial \dmn))^3$. Then,
\begin{equation*}
  \| u\|_{L^2(\partial\Omega)} \leq k^{-1} \|g\|_{L^2(\partial\Omega)} 
 \quad\text{and}\quad
 \| \sigma(u)\nu\|_{L^2(\partial\Omega)} 
   \leq \|g\|_{L^2(\partial\Omega)}.
 \end{equation*}
\end{lemma}
\begin{proof}
As in  
\cite{ChandlerwildeMonk2008,Esterhazy},
taking $v=u$ in \eqref{varformequation}, 
taking imaginary parts,
and using the impedance boundary condition
$\sigma(u)\cdot \nu = g-\ii ku$ on $\partial\Omega$
we arrive at 
\begin{equation*}
  k \|u\|_{L^2(\partial\Omega)}^2 
  = \Im \int_{\partial\Omega}g\cdot\bar{u}\,ds
  = \Im \int_{\partial\Omega}(\sigma(u) \nu+\ii ku)\cdot\bar{u}\,ds.
\end{equation*}
The first of these identities and the Cauchy inequality result in 
the first stated estimate.
The identity plus elementary manipulations furthermore show
\begin{equation*}
  0
  = \Im \int_{\partial\Omega}(\sigma(u)\nu)\cdot\bar{u}\,ds
  = -\frac{1}{k}
         \Re \int_{\partial\Omega} 
             (\sigma(u)\nu)\cdot(\overline{\ii k u})\,ds
\end{equation*}
Inserting the impedance boundary condition for $\overline{\ii k u}$
proves the second claimed estimate.
\end{proof}

For the next lemma, we will need the representation of the solution  
of \eqref{mainPDE} of the form \eqref{solutionrepresentation},
 that is often referred to as Somigliana's in this context.
Here, $\widetilde{V}_{k}$ and $\widetilde{K}_{k}$ are the layer 
potentials defined in Subsection~\ref{LayerSection}.

\begin{lemma}\label{l:GradEst}
 Let $\dmn\subset\,\mathbb{R}^3$ be a bounded
 open, connected Lipschitz domain. 
 Let $u\in (H^1(\dmn))^3$ be a solution to \eqref{mainPDE}
 with $f=0$ and $g\in (L^2(\partial\Omega))^3$. 
 Then, we have the following estimate
 \begin{align*}
  \norm{u}_{1,k}
  &\leq C k^{1/2} \|g\|_{L^2(\partial\Omega)}.
 \end{align*}
\end{lemma}
\begin{proof}
As in \cite{Spence2014}, we use  Somigliana's formula
\eqref{solutionrepresentation} and the following estimates 
for the layer potentials from Lemma~\ref{l:layerpotentials}
 \begin{align*}
 \|\widetilde{V}_k(\varphi)\|_{L^2(\Omega)}
 \leq C k^{-1/2} \|\varphi\|_{L^2(\partial\Omega)}
 \quad\text{and}\quad
  \|\widetilde{K}_k(\varphi)\|_{L^2(\Omega)}
 \leq C k^{1/2}\|\varphi\|_{L^2(\partial\Omega)}
\end{align*}
to conclude
\begin{equation*}
 \|u\|_{L^2(\Omega)}
 \leq C
 k^{-1/2} \|\sigma(u) \nu\|_{L^2(\partial\Omega)}
 +
 k^{1/2} \|u\|_{L^2(\partial\Omega)} .
\end{equation*}
Combining this with Lemma~\ref{l:cauchydata}
finally implies the $L^2$ bound
\begin{equation*}
 \|u\|_{L^2(\Omega)}
 \leq C  k^{-1/2} \|g\|_{L^2(\partial\Omega)}.
\end{equation*}
Relations \eqref{Garding} and \eqref{varformequation}
and $k\gtrsim 1$ with the Cauchy inequality
prove
\begin{equation*}
 C \|\nabla u\|_{L^2(\Omega)}^2
 \leq
 (k^2+c) \|u\|_{L^2(\Omega)}^2
 + 
 \Re F(u)
 \lesssim 
 k^2 \|u\|_{L^2(\Omega)}^2
 +
 \|g\|_{L^2(\partial\Omega)}
  \|u\|_{L^2(\partial\Omega)} .
\end{equation*}
Applying Lemma~\ref{l:cauchydata} and inserting the above
$L^2$ bound concludes the proof.
\end{proof}

We now state and prove our main theorem.  
\begin{theorem}\label{mainesttheorem}
  Let $\dmn\subset\,\mathbb{R}^3$ be a bounded 
  open, connected
  Lipschitz domain.
  Let $u\in (H^1(\dmn))^3$ be a solution to \eqref{mainPDE}, with 
  $g\in(L^2(\partial \dmn))^3$ and $f\in(L^2(\dmn))^3$. Then, 
  there exists a constant $C>0$ independent of $k$,
  such that
  \begin{align*}
   \norm{u}_{1,k} 
    \leq  C \left(k^{1/2}\TwoNorm{g}{\rob}+ k\TwoNorm{f}{\dmn}\right).
  \end{align*}
\end{theorem}
\begin{proof}
Taking the same approach as
 \cite{Esterhazy,Spence2014},
we transform the right 
hand side $f$ to the boundary.
We define 
$u_{0}=G_{k}*f$
(extending $f$ to zero outside $\dmn)$, 
where $G_{k}$ is the Green's function corresponding to \eqref{mainPDE}, 
written explicitly in Subsection~\ref{LayerSection}, \eqref{GreensFunction}.
We know from
Theorem~\ref{maintheorem}
that
\begin{align}\label{NewtonPotentialTheorem}
  k^{-1}\HHNorm{u_0}{\dmn}+ \HNorm{u_0}{\dmn}
 + k\TwoNorm{u_0}{\dmn}\leq C \TwoNorm{f}{\dmn}.
\end{align}
Writing $w=u-u_0$, we obtain
\begin{subequations}
\begin{align}
-\text{div}(\sigma(w))-k^2w&= 0 \text{ in } \dmn,\\
\sigma(w) \nu +\ii k w&=g-(\sigma(u_0) \nu+\ii k u_0)=:\tilde{g}
\text{ on } \rob. 
\end{align}
\end{subequations}
 Using the multiplicative trace inequality on $\tilde{g}$, we obtain
 \begin{align*}
\TwoNorm{\tilde{g}}{\rob}\leq &C\left(
\TwoNorm{g}{\rob}+ \norm{u_0}^{1/2}_{H^2(\dmn)}\norm{u_0}^{1/2}_{H^1(\dmn)}
+k\norm{u_0}^{1/2}_{H^1(\dmn)}\norm{u_0}^{1/2}_{L^2(\dmn)}\right)\\
\leq &C\left(\TwoNorm{g}{\rob}+k^{1/2}\TwoNorm{f}{\dmn}\right).
 \end{align*}
The combination of this bound with Lemma~\ref{l:GradEst} results in
 \begin{align*}
  \norm{w}_{1,k}
  &\leq C k^{1/2} \|\tilde g\|_{L^2(\partial\Omega)}
    \leq C \left(k^{1/2} \|\tilde g\|_{L^2(\partial\Omega)}
            + k \|f\|_{L^2(\Omega)} \right)
   .
 \end{align*}
 Including the Newton potential bounds \eqref{NewtonPotentialTheorem},
 we obtain our estimate.
\end{proof}

We now are in a position to derive the so-called inf-sup condition 
also derived in the same way as 
\cite{BrownGallistlPeterseim,Esterhazy}.
The proof will make use of the unique solvability in the case
of constant material coefficients,
see the discussion at the end of
\S\ref{ss:governing}.

\begin{theorem}\label{t:infsup}
 Let $\dmn\subset\,\mathbb{R}^3$ be a bounded 
  open, connected
  Lipschitz domain.
  Then there exists a $C>0$ independent of $k$, such that
  the variational form $\Phi$ in \eqref{varformequation}, satisfies
 \begin{align*}
  \inf_{0\neq u\in H^1(\dmn)^3}\sup_{0\neq v\in H^1(\dmn)^3}\frac{\Re \Phi(u,\bar v)}{\norm{u}_{1,k}\norm{v}_{1,k}}\geq Ck^{-2}.
 \end{align*}
\end{theorem}
\begin{proof}
 Let $u\in H^1(\dmn)^3$ be given.
As in \cite{Esterhazy} we define an 
 auxiliary solution $z\in H^1(\dmn)^3$ satisfying
 \begin{align}
\label{aboverelation}
 \Phi(w,\bar z)=2k^2(w,u)_{L^2(\dmn)}
 \quad\text{for all }w\in H^1(\Omega)^3,
 \end{align}
 and thus, by Theorem \ref{mainesttheorem},
$\norm{z}_{1,k}\leq Ck \TwoNorm{k^2 u}{\dmn}\leq Ck^3 \TwoNorm{u}{\dmn}$.
 Letting 
 $v=u+z$, and using the relation \eqref{aboverelation}, we obtain 
 \begin{align*}
 \Re\Phi(u,\bar v)
  = \Re\Phi(u,\bar u)+\Re\Phi(u,\bar z)
  =
  \Re\Phi(u,\bar u) + 2k^2(u,u)_{L^2(\dmn)}.
 \end{align*}
 Then, applying   Korn's second  inequality \eqref{korn_second},  we obtain 
\begin{align*}
\Re\Phi(u,\bar v)
=\Re\Phi(u,\bar u)+2\int_{\dmn}k^2 |u|^2dx
&=\int_{\dmn}   \left(\lambda|\text{div}(u)|^2+{2\mu}|\eps(u)|^2\right)dx
   +\int_{\dmn}k^2 |u|^2dx\\
&\geq c\norm{\nabla u}_{L^2(\Omega)}^2 +
     (k^2-C)   \norm{u}_{L^2(\Omega)}^2 
 \geq  C\norm{u}^2_{1,k},
\end{align*}
where again we used $C\lesssim k^2$. 
 Finally, using the bound 
$\norm{v}_{1,k}=\norm{u+z}_{1,k}\leq Ck^2 \norm{u}_{1,k}$,
 we have
 \begin{align*}
  \Re\Phi(u,v)\geq C \norm{u}^2_{1,k} 
 \geq 
 Ck^{-2}\norm{u}_{1,k}\norm{v}_{1,k}.
  \end{align*}
\end{proof}

\section{Multiscale Method}\label{s:multiscale}
In this section we describe the application of 
the multiscale Petrov--Galerkin method (msPGFEM or msPG method)
from \cite{BrownGallistlPeterseim,GallistlPeterseim2015,Peterseim2016}
to the 
elasticity setting.
This method is based on ideas in an  algorithm developed for
homogenization problems in
\cite{Brown.Peterseim:2014,HP13,MP14}
also known as Localized Orthogonal Decomposition.
The ideas have been adapted to the Helmholtz problem for
homogeneous coefficients in \cite{Peterseim2016},
and later presented in the Petrov--Galerkin framework  
\cite{BrownGallistlPeterseim,GallistlPeterseim2015,Peterseim2015survey}.

\subsection{Meshes and Data Structures}
We begin with the basic notation needed regarding the relevant mesh 
and data structures. We keep the presentation general and will link 
it back to the elastodynamics case as we proceed.
Let $\T_H$ be a shape-regular partition of $\Omega$ into 
intervals, parallelograms, parallelepipeds for
$d=1,2,3$, respectively, such that $\bigcup\T_H =\overline\Omega$
and any two distinct $T,T'\in\T_H$ are either disjoint or share
exactly one lower-dimensional hyper-face
(that is a vertex or an edge for $d\in\{2,3\}$ or a face
for $d=3$).
We suppose the mesh is quasi-uniform.
For simplicity, we are considering quadrilaterals (resp. hexahedra) with
parallel faces.

Given any subdomain $S\subset\overline\Omega$, we define its neighborhood 
to be
\begin{equation*}
\nei(S):=\operatorname{int}
          \Big(\cup\{T\in\T_H\,:\,T\cap\overline S\neq\emptyset  \}\Big).
\end{equation*}
Furthermore, we introduce for any $m\geq 2$ the patch extensions
\begin{equation*}
\nei^1(S):=\nei(S)
\qquad\text{and}\qquad
\nei^m(S):=\nei(\nei^{m-1}(S)) .
\end{equation*}
Note that the shape-regularity implies that there is a uniform bound denoted
$C_{\mathrm{ol},m}$,
on the number of elements in the $m$th-order patch,
$
\#\{ K\in\T_H\,:\, K\subset \overline{\nei^m(T)}\}
\leq C_{\mathrm{ol},m}
$
for all ${T\in\T_H}$.
We will abbreviate $C_{\mathrm{ol}}:=C_{\mathrm{ol},1}$.
The assumption that the coarse-scale mesh $\T_H$
is quasi-uniform implies that $C_{\mathrm{ol},m}$ depends polynomially
on $m$.
The global mesh-size is 
$H:=\max\{\operatorname{diam}(T): T\in\T_H\}$.

We will denote 
$Q_1(\T_H)$ to be  the space of piecewise bilinear.
degree less than or equal to $1$.
The space of globally continuous piecewise first-order polynomials 
is then given by
$
\mathcal{S}^1(\T_H):= C^0(\Omega)\cap Q_1(\T_H),
$
and by incorporating the Dirichlet condition we arrive at 
the standard $Q_1$ finite element space denoted here as 
\begin{equation*}
V_H
:=
[\mathcal{S}^1(\T_H)]^d \cap V
\end{equation*}
where $V$ denotes the energy space from
see \S\ref{ss:varsetting} below.
We use quadrilateral/hexa\-hedral elements
for the ease of presentation, but
the results remain valid for piecewise affine finite elements over
simplicial partitions.

To construct our fine-scale and, thus, multiscale spaces we will need
to define a coarse-grid quasi-interpolation operator. 
For simplicity of presentation,we suppose here that this 
quasi-interpolation is also projective.
We let  $I_H:V\to V_H$ be a surjective
quasi-interpolation operator that
acts as a stable quasi-local projection in the sense that
$I^2_H = I_H$ and that
for any $T\in\T_H$ and all $v\in V$ the following local stability 
result holds
\begin{equation}\label{e:IHapproxstab}
H^{-1}\|v-I_H v\|_{L^2(T)} + \|\nabla I_H v \|_{L^2(T)}
\leq C_{I_H} \|\nabla v\|_{L^2(\nei(T))} .
\end{equation}
Under the mesh condition that 
\begin{equation}\label{e:resolution}
k H \lesssim 1 
\end{equation}
is bounded by a generic constant,
this implies  stability in the $\|\cdot\|_V$ norm
\begin{equation}\label{e:IHapproxstabV}
\|I_H v\|_V \leq C_{I_H,V} \|v\|_V
\quad\text{for all } v\in V,
\end{equation}
with a $k$-independent constant $C_{I_H,V}$. 
One possible choice and which we use in our implementation of the method,
is to define $I_H:=E_H\circ\Pi_H$, where
$\Pi_H$ is the piecewise $L^2$ projection onto $Q_1(\T_H)$
and $E_H$ is the averaging operator that maps $Q_1(\T_H)$ to $V_H$ by
assigning to each free vertex the arithmetic mean of the corresponding
function values of the neighbouring cells.

\subsection{The Variational Setting}
\label{ss:varsetting}
Let $\Omega\subset\mathbb R^d$ for $d=2$ or $d=3$ be a bounded
polygonal Lipschitz domain with disjoint boundary portions
$\Gamma_R$, $\Gamma_D$, $\Gamma_N$
and define the energy space
$V:=\{v\in (H^1(\Omega))^d : v|_{\Gamma_D} = 0\}$
equipped with the norm
$\|\cdot\|_V 
 :=\sqrt{k^2\|\cdot \|_{L^2(\Omega)}^2 +\|\Ctensor^{1/2}\varepsilon \cdot \|_{L^2(\Omega)}^2}
$
(which is, by Korn's inequality, equivalent to
$\|\cdot\|_{1,k}$ from 
the foregoing section).
Here, we define the elasticity tensor $\Ctensor$ to act on a symmetric $d\times d$
matrix $M$ by double contraction as
$\Ctensor M = 2\mu M + \lambda \operatorname{tr}M I_{d\times d}$.
As equation \eqref{varformequation}, define on $V$ the sesquilinear form
\begin{equation}\label{e:DEFa}
a(v,w) := 
(\Ctensor \eps(v),\eps(w))_{L^2(\Omega)}
      -(k^2 v,w)_{L^2(\Omega)}
      + (\ii k v,w)_{L^2(\Gamma_R\cap \partial \Omega)}.
\end{equation}
For a given volume force $f\in L^2(\Omega)^3$ and Robin data
$g\in L^2(\Gamma_R)$,
the elasticity problem in variational form seeks $u\in V$ such that
\begin{equation*}
a(u,v) = (f, v)_{L^2(\Omega)}
  + (g, v)_{L^2(\Gamma_R)}
  \qquad\text{for all } v\in V.
\end{equation*}
For simplicity, we focus on homogeneous Dirichlet and Neumann data.
For the case that $\Gamma_R=\partial\Omega$,
\eqref{e:DEFa} corresponds to
\eqref{varformequation}, and
Theorem~\ref{t:infsup}
proves the stability condition
\begin{equation}\label{e:elasthelmholtzstability} 
\big(\gamma(k,\Omega)\big)^{-1}
\leq
 \inf_{v\in V\setminus\{0\}}
 \sup_{w\in V\setminus\{0\}}
 \frac{\Re a(v,w)}{\|v\|_V\| w\|_V} .
\end{equation}
where $\gamma(k,\Omega)$ depends polynomially
(at most quadratically)
on $k$.
For more general boundary configurations,
the polynomial growth of $\gamma(k,\Omega)$ in
\eqref{e:elasthelmholtzstability} will be imposed as an
\emph{assumption} throughout this numerical methods section.

\subsection{Definition of the Method}\label{ss:defMethod}

The multiscale method is determined by three parameters,
namely the coarse-scale mesh-size $H$, the fine-scale mesh-size $h$, 
and the oversampling parameter $m$. 
We assign to any $T\in\T_H$ its $m$-th order patch
$\Omega_T:=\nei^m(T)$, $m\in \mathbb{N}$, 
and define for any $v,w\in V$ the localized sesquilinear forms of  
\eqref{varform} (resp.\ \eqref{e:DEFa}) to $\Omega_T$ as 
\begin{equation*}
a_{\Omega_{T}}(u,v)
=(\Ctensor \eps(u),\eps(v))_{L^2(\Omega_{T})}
     -(k^2u ,v)_{L^2(\Omega_{T})}
     + (\ii ku,v)_{L^2(\partial\Omega_T\cap  \Gamma_R)}.
\end{equation*}
Similarly, for $T$ we have
\begin{equation*}
a_T(u,v)
=(\Ctensor \eps(u),\eps(v))_{L^2({T})}
      -(k^2 u ,v)_{L^2({T})}
      + (\ii k u,v)_{L^2(\Gamma_R\cap \partial {T})}.
\end{equation*}
Let the fine-scale mesh $\T_h,$ be a global uniform refinement 
of the mesh $\T_H$ over
$\Omega$ and define
\begin{equation*}
V_h(\Omega_T) 
 := \{ v\in [Q_1(\T_h)]^d \cap V\,: v=0\text{ outside }\Omega_T\} .
\end{equation*}
Define the null space of the quasi-interpolation operator $I_H$ by
\begin{equation*}
W_h(\Omega_T) := \{ v_h\in V_h(\Omega_T) \,:\, I_H(v_h) = 0\}.
\end{equation*}
This is the space often referred to as the fine-scale or 
small-scale space. 
Given any scalar nodal basis function $\Lambda_z$
and the vector basis function $\Lambda_z e_j \in V_H$,
$e_j\in\mathbb R^d$ denoting the $j$th Cartesian unit vector,
and 
let 
$\lambda_{z,T}^{(j)}\in W_h(\Omega_T)$
solve the subscale corrector problem
\begin{equation}\label{e:lambdacorrectorproblem}
a_{\Omega_T}(w,\lambda_{z,T}^{(j)}) = a_T(w,\Lambda_z e_j)
\quad\text{for all } w\in W_h(\Omega_T).
\end{equation}
Let $\lambda_z^{(j)}:=\sum_{T\in\T_H} \lambda_{z,T}^{(j)}$
and define the multiscale test function
\begin{equation*}
\widetilde\Lambda_z^{(j)} := \Lambda_z e_j-  \lambda_z^{(j)}.
\end{equation*}
The space of multiscale test functions then reads
\begin{equation*}
\widetilde V_H :=
 \operatorname{span}\{\widetilde\Lambda_z^{(j)}\,:\,
                   z\text{ free vertex in }\T_H,
                   \: j=\in\{1,\dots,d\} \} .
\end{equation*}
We emphasize that the dimension of the multiscale space is the same 
as the original coarse space,
$\dim V_H = \dim \widetilde V_H$. Moreover, the dimension
is independent of the 
parameters $m$ and $h$.
Finally, the multiscale Petrov--Galerkin FEM seeks to find 
$u_H\in V_H$ such that
\begin{equation}\label{e:discreteproblem}
a(u_H,\tilde v_H) = 
(f,\tilde v_H)_{L^2(\Omega)}
  + (g,\tilde v_H)_{L^2(\Gamma_R)}
\quad\text{for all } \tilde v_H\in \widetilde V_H.
\end{equation}

As in \cite{GallistlPeterseim2015}, the error analysis
shows that the choice $H\lesssim k^{-1}$, $m\approx\log(k)$
will be sufficient  to guarantee stability and 
quasi-optimality properties,
provided that $k^\alpha h\lesssim 1$ where $\alpha$ depends on the
stability and regularity of the continuous problem.
The conditions on $h$ are the same as for the standard $Q_1$ FEM on the
global fine scale.
For example, the stability analysis in this paper combined the arguments of
\cite{melenk1995generalized}  shows that in three space dimensions
$k^2 h\lesssim 1$
is sufficient for stability and quasi-optimality
for the case of pure Robin boundary conditions.
More generally,
if the corresponding adjoint problem
\begin{equation*}
 a(w,z)=(f,w)_{L^2(\Omega)}
 \qquad\text{for all }v\in V,
\end{equation*}
admits an elliptic regularity estimate of the form
$$
\|z\|_{H^{1+s}(\Omega)} \leq C \|f\|_{L^2(\Omega)}
\quad\text{for some } 0<s\leq 1
$$
and the stability bound
$$
\|u\|_V\lesssim k^\beta\|f\|_{L^2(\Omega)}
$$
holds in place of Theorem~\ref{mainesttheorem},
then it can be shown that the smallness condition reads
$hk^{(1+\beta)/s}\lesssim 1$
for our low-order discretization.
Such relations are well studied in the acoustic Helmholtz case,
see
\cite[Proof of Prop.\ 8.2.7]{melenk1995generalized},
\cite{Sauter2006},
and can be shown similarly for the elastic Helmholtz setting.

\subsection{Brief Error Analysis}

As the related error analysis of the method and the truncated method 
are well studied \cite{BrownGallistlPeterseim,GallistlPeterseim2015},
we give a brief overview of the main results and error estimates 
available to our multiscale method. The key point being that for our
method to remain pollution free and be computationally tractable, 
the solution must obey polynomial-$k$ growth. This is connected to 
our analysis of the inf-sup condition in Section \ref{kgrowth}.
The polynomial growth of $\gamma(k,\Omega)$ in
\eqref{e:elasthelmholtzstability} for the case of
a pure Robin boundary is verified in this paper.

Throughout this section we assume the natural resolution condition
\eqref{e:resolution}.

\begin{lemma}[well-posedness of corrector problems]
Provided $k H\lesssim$ 1, the corrector problem
\eqref{e:lambdacorrectorproblem} is well-posed.
We have for all $w\in W_h$ equivalence of norms
\begin{equation}\label{e:normequiv}
\| \nabla w\|_{L^2(\Omega)}
\lesssim
\| w\|_V
\lesssim
\| \nabla w\|_{L^2(\Omega)}.
\end{equation}
and coercivity
\begin{equation}\label{e:coercivity}
\|\nabla w\|_{L^2(\Omega)} \lesssim\Re a(w,w).
\end{equation}
\end{lemma}
\begin{proof}
The first inequality of \eqref{e:normequiv} follows from
Korn's inequality \eqref{korn_second}.
The second estimate follows from the interpolation estimate
\eqref{e:IHapproxstab} and the finite overlap of element patches.
Indeed, for any $w\in W_h$,
\begin{equation*}
k^2 \|w\|_{L^2(\Omega)}^2
=
k^2 \|(1-I_H) w\|_{L^2(\Omega)}^2
\lesssim
(k H)^2 \|\nabla w\|_{L^2(\Omega)}^2.
\end{equation*}
This estimate and Korn's inequality
yield for some constants $c$, $C$ that
\begin{equation*}
\Re a(w,w)
=
\| \Ctensor^{1/2} \varepsilon (w) \|_{L^2(\Omega)}
-k^2 \|w\|_{L^2(\Omega)}
\ge
c \|\nabla w \|_{L^2(\Omega)}
- C (k H)^2 \|\nabla w\|_{L^2(\Omega)}^2,
\end{equation*}
so that, for $(k H)^2 $ small enough, we conclude
the coercivity \eqref{e:coercivity}.
\end{proof}

Provided $h$ is chosen fine enough,
the standard FEM over $\T_h$ is stable in the sense that
there exists a constant $C_{\mathrm{FEM}}$
such that with $\gamma(k,\Omega)$ 
from \eqref{e:elasthelmholtzstability} 
there holds
\begin{equation}\label{e:finehstability}
 \big(C_{\mathrm{FEM}}\gamma(k,\Omega) \big)^{-1}
  \leq
  \inf_{v\in V_h\setminus\{0\}}\sup_{w\in V_h\setminus\{0\}} 
      \frac{\Re a(v,w)}{\|v\|_V\|w\|_V} .
\end{equation}
This is actually a condition on the fine-scale parameter $h$.
In general, the requirements on $h$ depend on the stability of the 
continuous problem.

The following two results follow as in
\cite{GallistlPeterseim2015,Peterseim2016}. Their proofs are
omitted for brevity.

\begin{theorem}[well-posedness of the discrete problem]\label{t:wellposed}
Under the resolution conditions \eqref{e:resolution}
and \eqref{e:finehstability}
and the following oversampling condition
\begin{equation}\label{e:mcondition}
m\geq 
C_1
\lvert\log\big(C \gamma(k,\Omega)\big)\rvert
\end{equation}
problem \eqref{e:discreteproblem} is well-posed and there is
a constant
$C_{\mathrm{PG}}$ satisfying
\begin{equation*}
\big(C_{\mathrm{PG}}\gamma(k,\Omega)\big)^{-1}
\leq
 \inf_{v_H\in V_H\setminus\{0\}}
 \sup_{\tilde v_H\in \widetilde V_H\setminus\{0\}}
 \frac{\Re a(v_H,\tilde v_H)}{\|v_H\|_V\|\tilde v_H\|_V} .
\end{equation*}
\end{theorem}

\begin{theorem}[quasi-optimality]\label{t:quasiopt}
 The resolution conditions \eqref{e:resolution} and
 \eqref{e:finehstability}
 and the oversampling
 condition \eqref{e:mcondition}
 imply that the solution $u_H$ to
 \eqref{e:discreteproblem} with parameters $H$, $h$, and $m$ 
 and the solution $u_h$ of the standard Galerkin FEM on the mesh
 $\T_h$ satisfy
\begin{equation*}
\|u_h - u_H\|_V
\lesssim
  \|(1-I_H) u_h \|_V
\approx 
  \min_{v_H\in V_H} \| u_h - v_H \|_V .
\end{equation*}
\end{theorem}

The following consequence of Theorem~\ref{t:quasiopt} states an
estimate for the error $u-u_H$.

\begin{corollary}
Under the conditions of Theorem~\ref{t:quasiopt}, the discrete
solution $u_H$ to \eqref{e:discreteproblem} satisfies with some
constant $C\approx 1$ that
\begin{equation*}
\| u- u_H \|_V
\leq
  \| u- u_h \|_V
  +
  C \min_{v_H\in V_H} \| u_h - v_H \|_V .
\end{equation*}
\qed
\end{corollary}

\subsection{Numerical Experiment in 3D}

We present a numerical experiment on the unit cube
$\Omega = (0,1)^3$ with Robin boundary.
The exact solution reads
$$
  u(x) = 
    \frac{1}{k^2 |x+q|}
      \begin{pmatrix}
               \exp(\ii k |x+q|) -1 \\
               \exp(-\ii k |x+q|) -1 \\
               \exp(\ii k |x+q|) -1
      \end{pmatrix}
\quad\text{for }
q = \begin{pmatrix}
          1 \\ 1\\ 1
          \end{pmatrix}
.
$$
The data $f$ and $g$ were computed according to the Lam\'e
coefficients $\mu=1=\lambda$.
We compare the msPG FEM with the standard $Q_1$ FEM 
for wavenumbers $k=16$ and $k=32$
on uniform meshes with mesh size $H=2^{-3},2^{-4},2^{-5}$.
The reference mesh size is $h=1/64$.
Figure~\ref{f:3Dconvergence_k16} compares the normalized errors
in the $\|\cdot\|_{L^2(\Omega)}$ norm and the $\|\cdot\|_V$ norm
for $k=16$.
Figure~\ref{f:3Dconvergence_k32} displays the corresponding 
results for $k=32$
of the $Q_1$ FEM and those of the msPG method with oversampling
parameters $m=1$ and $m=2$.
While the performance of the FEM is dominated by the pollution
effect, the msPG FEM yields accurate results, in particular
for $m=2$. For $k=32$, we observe resonance effects in the
error of the msPG method for meshes close to the resolution
$kH\approx 1$.
Figure~\ref{f:3Derror} displays slice plots of the pointwise
error for the FEM and the msPG method ($m=2$) for $k=32$
on the mesh $\T_H$ with $H=2^{-5}$.

\begin{figure}[pt]
\includegraphics[width=.49\textwidth]{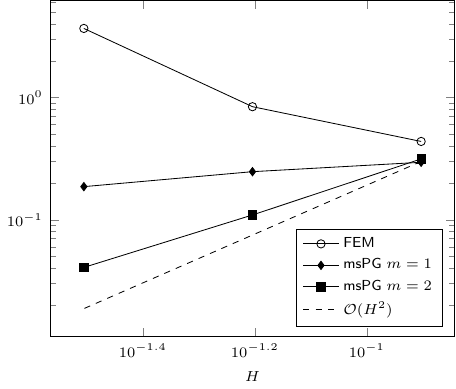}
\includegraphics[width=.49\textwidth]{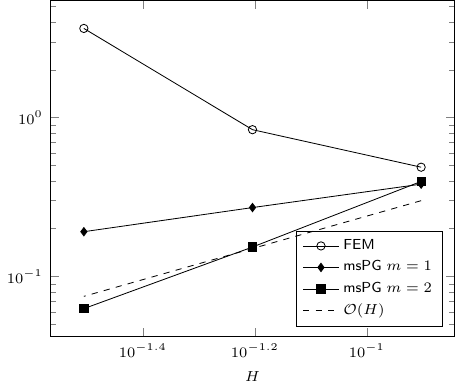}
\caption{%
   Normalized errors  $\|\cdot\|_{L^2(\Omega)}/\| u \|_{L^2(\Omega)}$
   (left) and  $\|\cdot\|_V/\| u \|_V$ (right)
   for the 3D example for $k=16$.
   \label{f:3Dconvergence_k16}
   }
\end{figure}

\begin{figure}[pt]
\includegraphics[width=.49\textwidth]{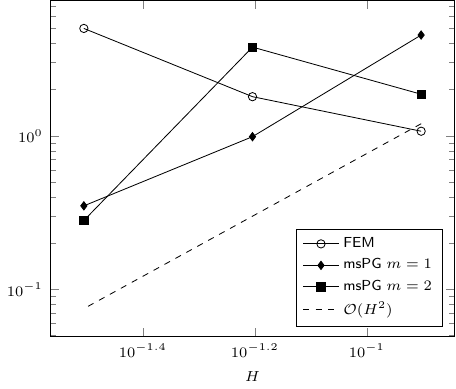}
\includegraphics[width=.49\textwidth]{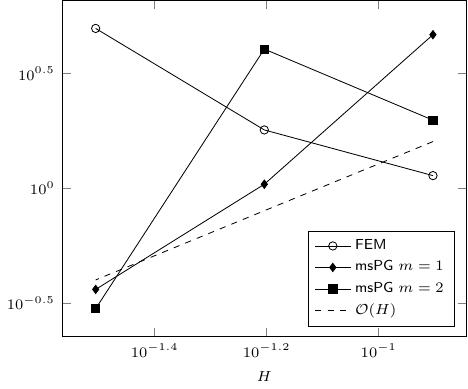}
\caption{%
   Normalized errors  $\|\cdot\|_{L^2(\Omega)}/\| u \|_{L^2(\Omega)}$
   (left) and  $\|\cdot\|_V/\| u \|_V$ (right)
   for the 3D example for $k=32$.   
   \label{f:3Dconvergence_k32}
   }
\end{figure}

\begin{figure}[pt]
\includegraphics[width=.49\textwidth]{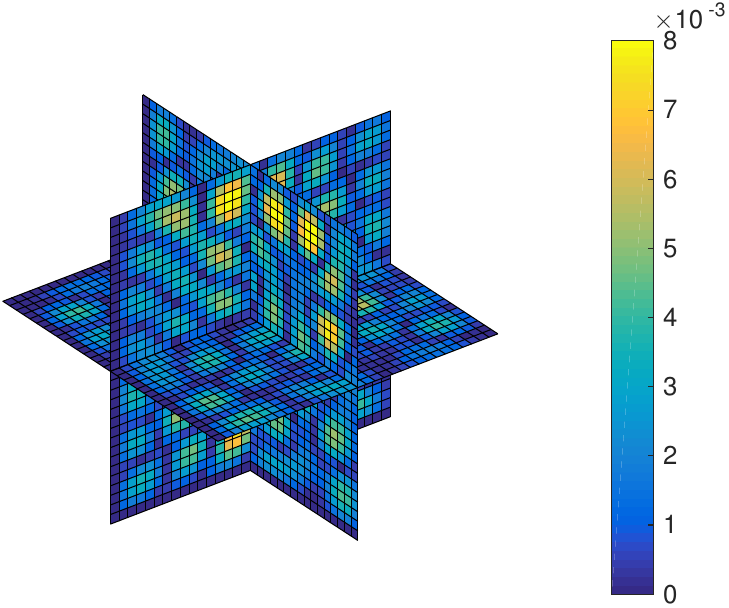}
\includegraphics[width=.49\textwidth]{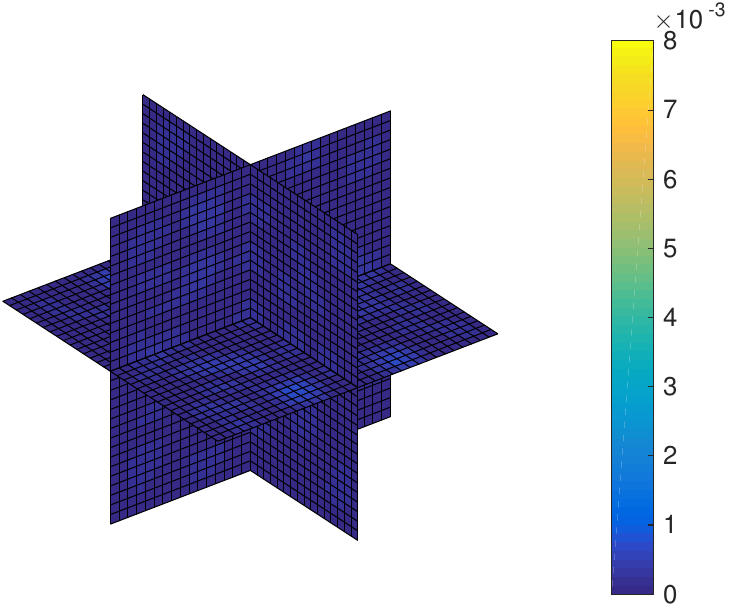}
\caption{%
   Slice plots of the modulus of the pointwise error 
   of the FEM (left) and the msPG method with $m=2$ (right)
   for the 3D example with $k=32$.
   \label{f:3Derror}
   }
\end{figure}

\subsection{Numerical Experiment in 2D}

We consider the square with hole
$\Omega = (0,1)^2\setminus [0.375,0.625]^2$
with Robin boundary conditions 
$\Gamma_R = \partial(0,1)^2$ on the outer boundary and
zero Dirichlet conditions 
$\Gamma_D = \partial[0.375,0.625]^2$ on the inner boundary.
The Robin data is $g=0$
while $f$ is the approximate point source with components
$$
f_j(x) =
\begin{cases}
\exp\left(-\frac{1}{1-(20|x|)^2}\right)& \text{for } |x| < 1/20\\
0 & \text{else}
\end{cases}
\qquad j=1,2.
$$
The Lam\'e parameters are $\mu=1=\lambda$.
The coarse meshes $\T_H$ have mesh sizes
$H=2^{-5},2^{-6},2^{-7},2^{-8}$
and the reference mesh size is $h=2^{-11}$.
Since the exact solution is unknown, we took the finite element
solution with respect to the fine-scale mesh $\T_h$ as a reference
solution.
We chose wavenumbers $k=64$ and $k=128$.
Figure~\ref{f:2Dconvergence_k64} displays the normalized errors in
the $\|\cdot\|_{L^2(\Omega)}$ norm and the $\|\cdot\|_V$ norm
for $k=64$ for the FEM and the msPG method with $m=1$ and $m=2$.
The errors for $k=128$ are shown in
Figure~\ref{f:2Dconvergence_k128}.
Figure~\ref{f:displacement2D} shows the elastic displacement computed
with the msPG method for $m=2$ and $H=2^{-7}$.
In all cases, the msPG approximation has optimal order
under the natural resolution condition
whereas the FEM suffers from pollution.

\begin{figure}[pt]
\includegraphics[width=.49\textwidth]{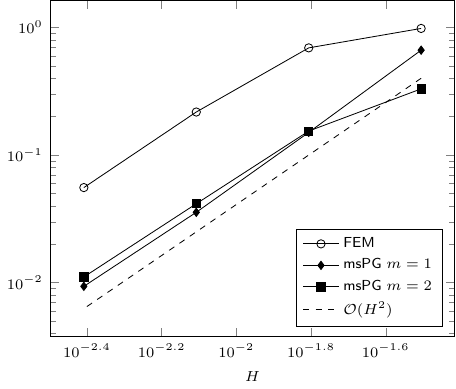}
\includegraphics[width=.49\textwidth]{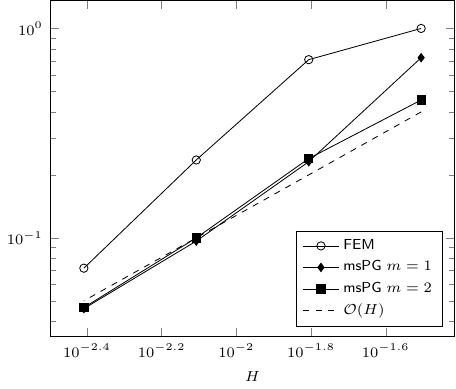}
\caption{%
   Normalized errors  $\|\cdot\|_{L^2(\Omega)}/\| u \|_{L^2(\Omega)}$
   (left) and  $\|\cdot\|_V/\| u \|_V$ (right)
   for the 2D example for $k=64$. 
   \label{f:2Dconvergence_k64}
   }
\end{figure}

\begin{figure}[pt]
\includegraphics[width=.49\textwidth]{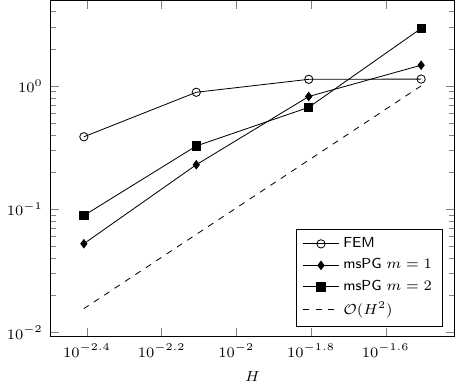}
\includegraphics[width=.49\textwidth]{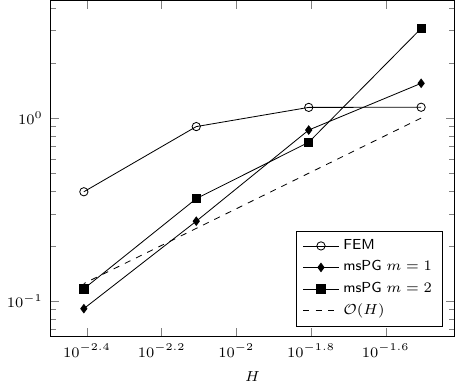}
\caption{%
   Normalized errors  $\|\cdot\|_{L^2(\Omega)}/\| u \|_{L^2(\Omega)}$
   (left) and  $\|\cdot\|_V/\| u \|_V$ (right)
   for the 2D example for $k=128$.
   \label{f:2Dconvergence_k128}
   }
\end{figure}

\begin{figure}[pt]
\centering
\includegraphics[width=.5\textwidth]{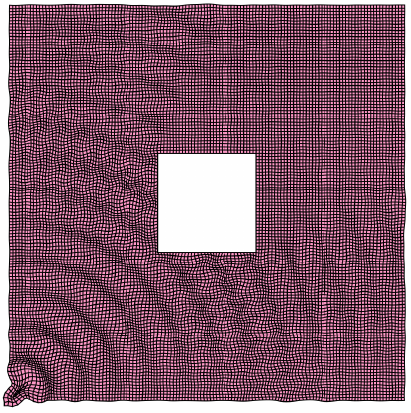}
\caption{%
   Elastic displacement in the 2D experiment for $k=128$ computed 
   with the msPG method with $m=2$ and $H=2^{-7}$;
   amplified by a factor 900.
   \label{f:displacement2D}
   }
\end{figure}

\appendix

\section{Newton Potential Estimates}\label{NPSection} 

In this appendix, we will estimate the Newton potential 
\eqref{newton.layerk}. 
 We utilize Fourier techniques 
 as in \cite{melenk2010convergence}, to 
 calculate the $k$-bounds on $N_k$.
 The main result is the following.

\begin{theorem}\label{maintheoremRestated}
Let $f\in (L^2(\dmn))^3,$  
for the Newton potential \eqref{newton.layerk} we have the estimate
\begin{align}\label{NPHTWORestated}
 k^{-1}\HHNorm{N_{k}^{}(f)}{\dmn}+ \HNorm{N_{k}^{}(f)}{\dmn}
 + k\TwoNorm{N_{k}^{}(f)}{\dmn}\leq C \TwoNorm{f}{\dmn},
\end{align}
where $C>0$ is independent of $k$ and depends only on 
 $\Omega,\mu,\lambda$. 
\end{theorem}

\begin{proof}
 According to the representation \eqref{GreensFunction.main}
 the Newton potential can be split
 as $N_k(f) = N_k^E(f) + N_k^H(f)$ where
 \begin{align*}
  N_k^E(f) (x) = \int_\Omega G_k^E(x-y) f(y) dy
 \quad\text{and}\quad 
  N_k^H(f)(x) = \int_\Omega G_k^H(x-y) f(y) dy.
 \end{align*}
The second part $N_k^H(f)$ is a vector version of the Newton
potential of the acoustic Helmholtz equation, for which the
bounds
$k^{1-m} \|N_k^H(f)\|_{H^m(\Omega)} \leq C \|f\|_{L^2(\Omega)}$
for $m=0,1,2$
have been established in \cite{melenk2010convergence}.
We therefore only need to prove
$k^{1-m} \|N_k^E(f)\|_{H^m(\Omega)} \leq C \|f\|_{L^2(\Omega)}$,
$m=0,1,2$, for the elastic part of the Newton potential.

We start by defining the auxiliary potential
\begin{align}\label{e:tildeNdef}
\widetilde N_k^E(f)(x)
=
\int_\Omega \widetilde G_k^E(x-y)  f(y) dy 
\qquad\text{where}\quad
\widetilde G_k^E(r) = 
\frac{1}{k^2}
\left(\frac{e^{\ii k_2 r}}{r}
     -\frac{e^{\ii k_1 r}}{r}\right)
\end{align}
and observe from the representation 
\eqref{GreensFunction.different} that
\begin{align*}
 |  N_k^E(f)|_{H^m(\Omega)}
 \leq 
C
| \widetilde N_k^E(f)|_{H^{m+2}(\Omega)}.
\end{align*}
The simplification in working with
$\widetilde G_k^E$ is that, in contrast to
$G_k^E$, it only depends on the radial component.
We proceed with a cut-off function argument and using Fourier 
techniques as in \cite{MelenkBIO}. 
Suppose $\dmn \subset B_{R}(0)$ for some $R>0$. 
We extend $f$ to zero when considered outside of $\dmn$ 
into $B_{R}$, but do not relabel. 
We define the cutoff function 
$\eta\in C^{\infty}(\mathbb{R}_{\geq 0})$,
such that $ \operatorname{supp}(\eta) \subset [0,4R]$
\begin{align*}
 \forall x\in \mathbb{R}_{\geq 0}:
   \quad 
 0\leq \eta(x)\leq 1,
 \quad
  |\eta'(x) |\leq C/R,
  \quad
 \eta|_{[0,2R]}=1,
 \quad
 \eta|_{[4R,\infty)}=0.
\end{align*}
We set $M(x):=\eta(|x|)$ and define an augmented Newton potential 
of \eqref{e:tildeNdef} as 
\begin{align}\label{e:augmentedpotential}
 v^{\eta}(x)=\int_{B_{R}(0)}M(x-y) \widetilde G_{{k}}^E(x-y)f(y)dy
\quad
  \text{for } x\in \mathbb{R}^3,
\end{align}
where 
$\widetilde G_k^E$ is given by \eqref{e:tildeNdef}.
 For functions $u$ with compact support, recall the Fourier transform
and the inverse transform are given 
for $x,\xi\in\mathbb R^3$ by 
\begin{align*}
\hat{u}(\xi)=\frac{1}{(2\pi )^{3/2 }}\int_{\mathbb{R}^3}e^{-\ii x\cdot \xi} u(x) dx
\quad\text{ and }\quad
{u}(x)=\frac{1}{(2\pi )^{3/2 }}\int_{\mathbb{R}^3}e^{\ii x\cdot \xi}
   \hat{u}(\xi) d\xi.
\end{align*}
For  $f$ has support in $B_R$ we  write  the truncated Newton 
potential component-wise, using the Einstein summation convention,
as  
$v_i^{\eta}(x)=((M\widetilde G_{k}^E)_{ij}*f_j)(x)$, 
for $i,j=1,2,3$. 
Taking the Fourier transform, using the standard convolution identity,
we obtain
$\hat{v}_i^{\eta}(\xi)
=(2\pi)^{3/2}(\widehat{(M\widetilde G_{k}^E)}_{ij}\hat{f}_j)(\xi)$,
for $i,j=1,2,3$. 
 For a multi-index  $\alpha\in \mathbb{N}_{0}^3$, 
(non-negative integer vectors of dimension 3), we denote
the corresponding multi-index derivatives as $\partial^{\alpha}$ 
in the standard way.
For the corresponding to the derivatives in the Fourier variable, 
we denote 
the function $P_{\alpha}:\mathbb{R}^3\to \mathbb{R}^3$, 
$P_{\alpha}(\xi)=\xi^{\alpha}$. 
For $2\leq |\alpha|\leq 4$, we see using the Plancherel identity that 
\begin{equation*}
\begin{aligned}
 \TwoNorm{\partial^{\alpha}v_i^{\eta}}{\mathbb{R}^3}
  &=\TwoNorm{P_{\alpha}\hat{v}_i^{\eta}}{\mathbb{R}^3}
  =
 (2\pi )^{3/2 }
   \TwoNorm{
    P_{\alpha}\widehat{(M\widetilde G_{k}^E)}_{ij}\hat{f}_j}{\mathbb{R}^3}
 \\
 &\leq  (2\pi )^{3/2 }
  \sup_{\xi\in\mathbb{R}^3}
   \abs{P_{\alpha}(\xi)\widehat{(M \widetilde G_{k}^E)}_{ij}(\xi)}
   \TwoNorm{\hat{f}_j}{\mathbb{R}^3}.
\end{aligned}
\end{equation*}
Thus, our estimate relies on the estimation of the supremum over 
$\xi$ on the last term.
This will be estimated in
Lemma~\ref{technicalGreenLemma} below, where we prove
that
\begin{align*}
  \sup_{\xi\in\mathbb{R}^3}
   \abs{P_{\alpha}(\xi)\widehat{(M \widetilde G_{k}^E)}_{ij}(\xi)}
 \leq 
 C k^{|\alpha|-3}.
\end{align*}
This implies the asserted estimate. 
\end{proof}

 We now state and prove our main technical lemma used in the proof of 
 Theorem~\ref{maintheorem}.

\begin{lemma}\label{technicalGreenLemma}
Let $(\widetilde G_{k}^E)_{ij}$ 
be given by \eqref{e:tildeNdef} and $M$ a cutoff function as above.
Then, there exists a $C>0$ depending only on $R,\mu,\lambda$ and not 
on $k$, so that for 
$|\alpha|=2,3,4$,
\begin{align*}
\sup_{\xi\in\mathbb{R}^3}
 \abs{P_{\alpha}(\xi)\widehat{(M\widetilde G_{k}^E)}_{ij}(\xi)}
\leq C k^{|\alpha|-3},
 \end{align*}
 for $i,j=1,2,3$.
\end{lemma}

\begin{proof}
We proceed as in \cite[Lemma 3.7]{melenk2010convergence}.
Since $G_k^E$ from \eqref{e:tildeNdef} depends only on
the radial component we can write
\begin{align}\label{e:gradial}
 \widetilde G_k^E(r) 
= 
g_k(r) := \frac{1}{k^2}
\left(\frac{e^{\ii k_2 r}}{r}
     -\frac{e^{\ii k_1 r}}{r}\right)
\quad\text{where }r=|x|.
\end{align}
A key observation is that $\widetilde G_k^E(r)\to0$ as $r\to0$, which corresponds to the boundary terms of the following integration by parts vanishing, allowing for higher order derivative estimates.
We then have from the definition of the Fourier transform
and a change of variables to spherical coordinates
$$
\widehat{M\widetilde G_k^E}(\xi)
=
\frac{1}{(2\pi )^{3/2 }}
 \int_{\mathbb{R}^3}e^{-\ii x\cdot \xi} M(x)\widetilde G_k^E(x) dx
=
\frac{1}{(2\pi )^{3/2 }}
 \int_0^\infty  \eta(r) g_k(r) r^2
  \left(\int_{S^2} e^{-\ii \zeta\cdot \xi} ds_\zeta\right)
dr
$$
where $S^2$ is the unit sphere in $\mathbb R^3$.
The inner integral was explicitly computed in
\cite[equation $($3.34$)$]{melenk2010convergence}
and equals
$ 4\pi \sin(r|\xi|)/(r|\xi|)$.
We thus obtain
\begin{align*}
 \widehat{M\widetilde G_k^E}(\xi)
=
\frac{4\pi}{(2\pi )^{3/2 }}
\iota(|\xi|)
\quad\text{with }
\iota(s)
:=
 \int_0^\infty  \eta(r) g_k(r) r^2
  \frac{\sin(rs)}{rs} dr
.
\end{align*}
We closely follow the arguments of
\cite[Lemma~3.7]{melenk2010convergence}
and estimate $s^m \iota(s)$ for $m=2,3,4$.
For $m=2$ we use the representation
\eqref{e:gradial} and integration by parts and compute
\begin{align*}
 |s^2\iota(s)|
&=
k^{-2}
\left|
\int_0^\infty
 \eta(r) (e^{\ii k_2 r}-e^{\ii k_1 r}) s \sin(rs)
dr
\right|
=
k^{-2}
\left|
\int_0^\infty
 \eta(r) (e^{\ii k_2 r}-e^{\ii k_1 r}) \partial_r\cos(rs)
dr
\right|
\\
&=
k^{-2}
\left|
\int_0^\infty
 \partial_r\Big(\eta(r) (e^{\ii k_2 r}-e^{\ii k_1 r}) \Big)\cos(rs)
dr
\right|
\\
&
\leq
k^{-2}
\left|
\int_0^\infty
  \eta'(r) (e^{\ii k_2 r}-e^{\ii k_1 r})\cos(rs)
dr
\right|
+
k^{-2}
\left|
\int_0^\infty
 \eta(r) (k_2 e^{\ii k_2 r}-k_1e^{\ii k_1 r})\cos(rs)
dr
\right| .
\end{align*}
By using the properties of $\eta$ and $\eta'$,
the first term can be bounded by $k^{-2}C$
and the second one by $k^{-1} CR $,
so that
$|s^2\iota(s)|\leq k^{-1} C$.
For $m=3$, in a similar fashion we use integration by parts
twice and obtain
\begin{align*}
 |s^3\iota(s)|
&
=
k^{-2}
\left|
\int_0^\infty
 \eta(r) (e^{\ii k_2 r}-e^{\ii k_1 r}) s^2 \sin(rs)
dr
\right|
\\
&
=
k^{-2}
\left|
\int_0^\infty
 \eta(r) (e^{\ii k_2 r}-e^{\ii k_1 r}) \partial^2_r\sin(rs)
dr
\right|
\\
&
=
k^{-2}
\left|
\int_0^\infty
 \partial^2_r\Big(\eta(r) (e^{\ii k_2 r}-e^{\ii k_1 r}) \Big)\sin(rs)
dr
\right|.
\end{align*}
With arguments analogous to above we see that this is bounded
by a constant $C$.
Finally, for $m=4$ we compute
\begin{align*}
 |s^4\iota(s)|
=
k^{-2}
\left|
\int_0^\infty
 \eta(r) (e^{\ii k_2 r}-e^{\ii k_1 r}) \partial^3_r\cos(rs)
dr
\right|
=
k^{-2}
\left|
\int_0^\infty
 \partial^3_r\Big(\eta(r) (e^{\ii k_2 r}-e^{\ii k_1 r}) \Big)\cos(rs)
dr
\right|
\end{align*}
and see that this is bounded by $Ck$.

In summary, we have shown $|s^m\iota(s)|\leq C k^{m-3}$.
This implies for $2\leq|\alpha|\leq 4$ that
\begin{align*}
\sup_{\xi\in\mathbb{R}^3}
 \abs{P_{\alpha}(\xi)\widehat{(M\widetilde G_{k}^E)}_{ij}(\xi)}
\leq C
\sup_{s\geq 0}
 |s^{|\alpha|} \iota(s) |
\leq C k^{|\alpha|-3},
 \end{align*}
which is the desired bound.
\end{proof}

\begin{acknowledgement}
We thank two anonymous referees who  helped to 
obtain a sharper-in-$k$ stability
bound and who pointed us to a much more direct 
argument
for proof of the stability result compared with a prior manuscript
version of this paper.
\end{acknowledgement}

\begin{funding}
  D.~Gallistl gratefully acknowledges financial support by the 
       DFG through SFB 1173;
       by the 
       Baden-W\"urttemberg Stiftung
       through the project
       ``Mehrskalenmethoden f\"ur Wellenausbreitung in 
       heterogenen Materialien und Metamaterialien'';
       and by the European Research Council
       trough project \emph{DAFNE}, ID 891734.
\end{funding}

\bibliographystyle{abbrv}
\bibliography{ElasticEquation}

\end{document}